\documentclass{amsart}
\usepackage[utf8]{inputenc}
\usepackage[T1]{fontenc}
\usepackage{amsmath}
\usepackage{amsfonts}
\usepackage{amssymb}
\usepackage{graphicx}
\usepackage{amsthm}

\theoremstyle{plain}
\newtheorem{lemma}{Lemma}[section]

\newtheorem{example}[lemma]{Example}
\newtheorem{remark}[lemma]{Remark}

\newtheorem{theorem}[lemma]{Theorem}
\newtheorem{corollary}[lemma]{Corollary}
\newtheorem{notation}[lemma]{Notation}
\newtheorem{definition}[lemma]{Definition}
\def\ov{\overline}
\def\eps{\varepsilon}
\def\R{\mathbb R}

\begin{document}

\title{Intrinsic character of Stokes matrices }

\author[J.F. Gagnon, C. Rousseau]{Jean-Fran\c{c}ois Gagnon and Christiane Rousseau}
\address{Jean-Fran\c{c}ois Gagnon, D\'epartement de math\'ematiques, Coll\`ege Montmorency, 475, boulevard de l'Avenir, Laval (Qc), H7N 5H9, Canada; Christiane Rousseau, D\'epartement de
math\'ematiques et de statistique\\Universit\'e de Montr\'eal\\C.P. 6128,
Succursale Centre-ville, Montr\'eal (Qc), H3C 3J7, Canada.}
\email{jean-francois.gagnon@cmontmorency.qc.ca, rousseac@dms.umontreal.ca}

\thanks{This work was supported by NSERC in Canada.}

\date{\today}

\maketitle

\begin{abstract}
Two germs of linear analytic differential systems $x^{k+1}Y^\prime=A(x)Y$ with a non resonant irregular singularity are analytically equivalent if and only if they have the same eigenvalues and equivalent collections of Stokes matrices. The Stokes matrices are the transition matrices between sectors on which the system is analytically equivalent to its formal normal form. Each sector contains exactly one separating ray for each pair of eigenvalues. A rotation in $S$ allows supposing that $\mathbb R^+$ lies in the intersection of two sectors. Reordering of the coordinates of $Y$ allows ordering the real parts of the eigenvalues, thus yielding triangular Stokes matrices. However, the choice of the rotation in $x$ is not canonical. In this paper we establish how the collection of Stokes matrices depends on this rotation, and hence on a chosen order of the projection of the eigenvalues on a line through the origin.
\end{abstract}

\section{Introduction}

Consider a germ of linear analytic differential system
\begin{equation}
\label{equation}
x^{k+1}Y^\prime=A(x)Y,\;\;\;x\in\mathbb{C},\;Y\in\mathbb{C}^n,
\end{equation}
with a non resonant irregular singularity of Poincar\'e rank $k$ at $0$. Then $A(x)$ is a matrix of germs of holomorphic functions at the origin and  the eigenvalues of $A(0)$ are distinct. Without loss of generality we can suppose that $A(0)$ is diagonal. There exists a  unique formal normalizing series tangent to the identity $Y=\hat{H}(x)Z= (\mathrm{id} + O(x))Z$ bringing \eqref{equation} to the diagonal normal form
\begin{equation}
\label{normal_form}
x^{k+1} Z^\prime= (D_0+ D_1x+ \dots+D_kx^k)Z,
\end{equation}
with $D_i$ diagonal and $D_0=A(0)$. The normal form has a canonical diagonal fundamental matrix solution that we call $W(x)$. 
However, generically, the normalizing series $\hat{H}$ is divergent.

Nevertheless, there exists $2k$ sectors $S_j$ of opening greater than $\frac{\pi}{k}$ (see Figure~\ref{sectors}) on which there exist unique normalizing holomorphic functions $H_j(x)$ that are asymptotic to $\hat H(x)$ on $S_j$. This defines a fundamental matrix solution  $W_j(x)= H_j(x)W(x)$ of \eqref{equation} over each $S_j$.
\begin{figure}\begin{center}
\includegraphics[width=6cm]{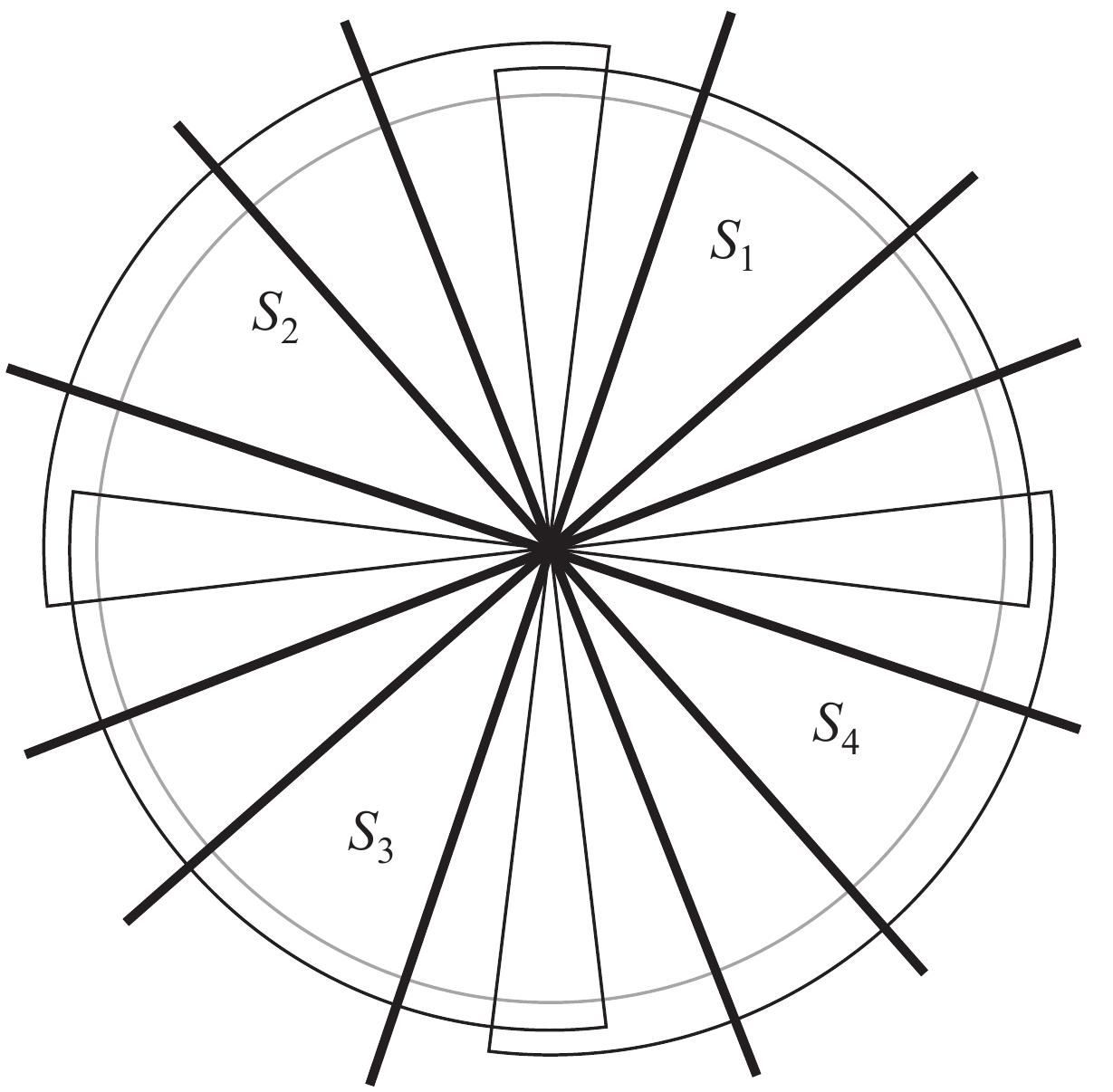}\caption{Four sectors when $k=2$. The bold lines are the separating rays.} \label{sectors}\end{center}\end{figure}
In the abundant literature on the subject (see for instance \cite{S}, \cite{IY} and \cite{R2}), it is often assumed that the eigenvalues of $A(0)$ satisfy the following inequality, a hypothesis that can be realized by means of a rotation in $x$ and a permutation of the coordinates in $Y$.
\begin{equation}
\label{hypothesis}
\mathfrak{R}(\lambda_1)>\mathfrak{R}(\lambda_2)>\dots>\mathfrak{R}(\lambda_n).
\end{equation}
Under this hypothesis, the columns $\{w_{1,j}, \dots, w_{n,j}\}$ of each $W_j$, which form a basis of the solution space,  are ordered with respect to flatness:
\begin{equation}\begin{cases}
w_{1,j'} \prec \dots \prec w_{n,j'}, & \mathrm{on} \: S_{2j}\cap S_{2j+1},\quad \mathrm{for} \:j'=2j,2j+1,\\
w_{1,j'} \succ \dots \succ w_{n,j'}, & \mathrm{on} \: S_{2j-1}\cap S_{2j},\quad \mathrm{for} \:j'=2j-1,2j,\end{cases} \label{order_flatness} \end{equation}
where indices are $\text{mod} 
\: 2k$. This comes from the fact that \begin{equation} w_{\ell,j}(x)= \exp\left(-\frac{\lambda_\ell}{kx^k}\right) v_{\ell,j}(x)\label{asympt_expansion}\end{equation} for some vector function $v_{\ell,j}(x)=O(1)$ on $S_j$. On the intersection $S_j\cap S_{j+1}$, the bases represented by $W_j$ and $W_{j+1}$ coexist and are linked by a matrix $C(j)\in GL(n,\mathbb C)$:
\begin{equation}
W_{j+1}=W_jC(j),\end{equation}
where indices are modulo $2k$.

The $C(j)$ are called \emph{Stokes matrices}. Generically, more precisely when $\hat{H}$ is divergent, some of the $C(j)$ are not diagonal. This is called the \emph{Stokes phenomenon}: the Stokes matrices measure the obstruction to have \eqref{equation} analytically equivalent to its normal form. Because of \eqref{order_flatness} we have that $C(j)$ is upper (resp. lower triangular) for $j$ even (resp. odd). 

When $x$ sweeps  a sector $S_j$, with increasing argument, the relative order of flatness of the $w_{\ell,j}$ changes. The change occurs on the \emph{separating rays}, which are the half-lines  determined by the condition $\mathfrak{R}\left(\frac{\lambda_\ell-\lambda_{\ell'}}{x^k}\right)=0$. Hence, there are $2k$ separating rays for each pair of eigenvalues $(\lambda_\ell,\lambda_{\ell'})$, one in each sector $S_j$. Of course, several pairs of eigenvalues can have the same separating rays. 
\\

In the presentation above, the choice of a rotation in $x$ corresponds to choosing a \emph{starting ray} $e^{i\theta} \mathbb R^+$ in $x$-space so that all eigenvalues have distinct projections on $e^{ik\theta}\mathbb R$. We say that the direction $e^{ik\theta}\R^+$ is \emph{non critical} in the eigenvalue space. When this direction is $\R^+$, the sector $S_1$ is chosen so that all separating rays inside $S_1$ have positive arguments. Hence, when starting on $\R^+$, we cross them when we turn in the positive direction.  This choice is non canonical. We could have chosen another non critical direction.

The starting ray $e^{i\theta} \mathbb R^+$ in $x$-space yields an order of the projections of the eigenvalues on the line $e^{ik\theta}\mathbb R$ oriented in the direction of $e^{ik\theta}$,  (which will induce an order of flatness  on the $\exp\left(-\frac{\lambda_\ell}{kx^k}\right)$ on $e^{i\theta}\R^+$), and the half-line $e^{i\theta}\mathbb R^+$ is called a \emph{non separating ray}. The coordinates of $Y$ are then permuted so that the order of flatness is in decreasing order. As mentioned above, there is no canonical way of choosing $\theta$. The \emph{separating rays} are the directions $e^{i\phi}$, for which $\mathfrak{R}\left(\left(\lambda_\ell-\lambda_{\ell'}\right)e^{-k\phi}\right)=0$. They divide the set of non separating rays $e^{i\theta}\mathbb R^+$ into a finite number of connected components. When constructing a sector $S_j$ containing a starting ray $e^{i\theta}\R^+$, we enlarge it with increasing argument, until it contains exactly one separating ray for each pair of eigenvalues. The other sectors are built in the same way with starting rays $e^{i\left(\theta+\frac{\ell \pi}{k}\right)}\R^+$, $\ell=1,\dots, 2k-1$.  We describe how the collection of Stokes matrices associated to a starting ray $e^{i\theta}\mathbb R^+$ changes when we cross a separating ray. The change is nontrivial. On a separating ray $e^{i\phi}\R^+$, some blocks of consecutive eigenvalues have identical projections on the critical ray $e^{ik\phi}\R^+$. The orders of projections of eigenvalues of each block are opposite on the two sides of the critical  ray. 
If we are considering an upper (resp. lower triangular Stokes matrix) $C$, then the new upper (resp. lower) Stokes matrix constructed after crossing the separating ray  is 
obtained as $P^{-1}U^{-1}CVP$, where $U$ and $V$ are block diagonal matrices with blocks of the lower (resp. upper) adjacent Stokes matrices on both sides of $C$, for the eigenvalues that have changed order, and identity blocks elsewhere, and where $P$ is a permutation matrix representing the new order of eigenvalues. The precise statement will be given below after we have introduced the necessary notations. We illustrate the theorem on an example in $\mathbb C^3$ for $k=1$.

\section{The main theorem}

Before stating the theorem, let us introduce some notation adequate to our purpose. Indeed, we will need to change the order of the eigenvalues in all subsets of eigenvalues that project on a unique point on a critical ray $e^{i\psi} \mathbb{R}^+$.
\begin{notation} 
\begin{enumerate}
\item $I_\ell$ and $J_\ell$ represent respectively the $\ell\times \ell$ identity  matrix and the matrix with $1$ on the anti-diagonal and $0$ elsewhere. 
\item Let $n= s_1+ r_2+ s_3+ r_4+ \dots+ s_{2m-1}+ r_{2m}+s_{2m+1}$ with $s_{2i+1}\in \mathbb N$ and  $r_{2i}\geq 2$.
We let 
\begin{equation}P_{s_1,r_2, \dots, s_{2m+1}}=\mathrm{diag}(I_{s_1}, J_{r_2}, \dots, J_{r_{2m}}, I_{s_{2m+1}}).\label{def_P} \end{equation} Calling $\overline{m}$ the ordered generalized partition of $n$ given by \begin{equation}\overline{m}= (s_1,r_2, \dots, s_{2m+1}),\label{def:m}\end{equation} we will also use the shortened notation $P_{\overline{m}}$.  (Note that $P_{\overline{m}}=P_{\overline{m}}^{-1}$.)
\end{enumerate}
\end{notation}

\begin{definition} \begin{enumerate}
\item Let $f$ and $g$ be meromorphic functions on a neighbourhood of $0$ and $R$ be an open ray (i.e. $R=e^{i\theta}\mathbb R^+$). We say that $f$ is \emph{flatter} than $g$ on $R$, and write $f\prec g$, or $g\succ f$, if $f/g\rightarrow0$ as $x\rightarrow0$ along $R$. If $S$ is a sector, then $f\prec g$ on $S$ if it is the case for every ray in $S$. 
\item Similarly, let $w(x)= (w^1(x), \dots, w^n(x))$ and $\ov{w}(x)=(\ov{w}^1(x), \dots, \ov{w}^n(x))$  be two vectors, the coordinates of which are holomorphic on a sector $S$. We say that $w$ is flatter than $\ov{w}$ on $R$ (resp. $S$), and write $w\prec \ov{w}$ if, for all $\ell$, $w^\ell\prec \ov{w}^\ell$ on $R$ (resp. $S$).\end{enumerate} \end{definition}

\begin{definition} In a system \eqref{equation} with $A(0)= \mathrm{diag}(\Lambda)$, where $\Lambda=(\lambda_1,...,\lambda_n)$, the \emph{separation rays} are given by the solutions to
\begin{eqnarray}
\mathfrak{R}\left(\frac{\lambda_p-\lambda_q}{x^{k-1}}\right)=0.
\end{eqnarray}\end{definition}

\begin{definition} A ray $e^{i\psi}\R^+$ is a \emph{critical ray} if several eigenvalues have equal projections on the line $e^{i\psi}\R$.\end{definition}

\begin{remark} If  $e^{i\psi}\R^+$  is a critical ray, then
$e^{i\left(\frac{\psi+j\pi}{k}\right)}\R^+$, $j=0, \dots, 2k-1$ are its associated separating rays. 
The critical rays are in the complex plane of eigenvalues, while the separating rays are in the $x$-plane. In particular, a critical ray is a separating ray when $k=1$.\end{remark}

Along the separating rays, the order of the solutions given by their respective order of flatness changes, and this happens nowhere else. This means that in a sector containing none of these rays, the ordering of solutions by their flatness is constant. Also, since $n$ is finite, is it possible to enumerate the separating rays as  $R_1,R_2,...,R_{2ku}$, where $R_j$ has argument $\phi_j\in[0,2\pi)$   and the $\phi_j$ are increasing. Note that hypothesis \eqref{hypothesis} implies that $\mathbb{R}^+$ is not a separating ray. Therefore it is used as a starting point to build the Stokes matrices.

\begin{definition} Let $R=e^{i\phi}\R^+$ be a ray  and $\mathrm{pr} (\lambda_j)$ be the signed length of the projection of the eigenvalue $\lambda_j$ on its associated  ray  $\ov{R}=e^{ik\phi}\R^+$ (ie. $\mathrm{pr}(\lambda_j)=\mathfrak{R}(\lambda_je^{-ik\phi})$.) We say that the \emph{order of eigenvalues on the ray $R$} is given by $\overline{m}= (s_1,r_2, \dots, s_{2m+1})$ if the subsets of eigenvalues corresponding to indices $r_j$ have equal projections, more precisely: 
$$\begin{cases} \mathrm{pr}(\lambda_j)\geq 
\mathrm{pr}(\lambda_{j+1}), &\text{for all} \quad j,\\
 \mathrm{pr}(\lambda_j)= 
\mathrm{pr}(\lambda_{j+1})& \text{if and only if } \quad j\in \sum_{i=1}^\ell r_{2(i-1)} +\sum_{i=1}^{\ell}s_{2i-1}+ [1, r_{2\ell}-1],\end{cases}$$ (see Figure~\ref{projection}).
Note that when $R$ is not a separating ray, then $m=0$ and $s_1=n$. Also, the order of eigenvalues corresponds to the respective order of flatness of the $\exp\left(-\frac{\lambda_j}{x^k}\right)$ on $R$.\end{definition}
 \begin{figure}\begin{center}
\includegraphics[width=7cm]{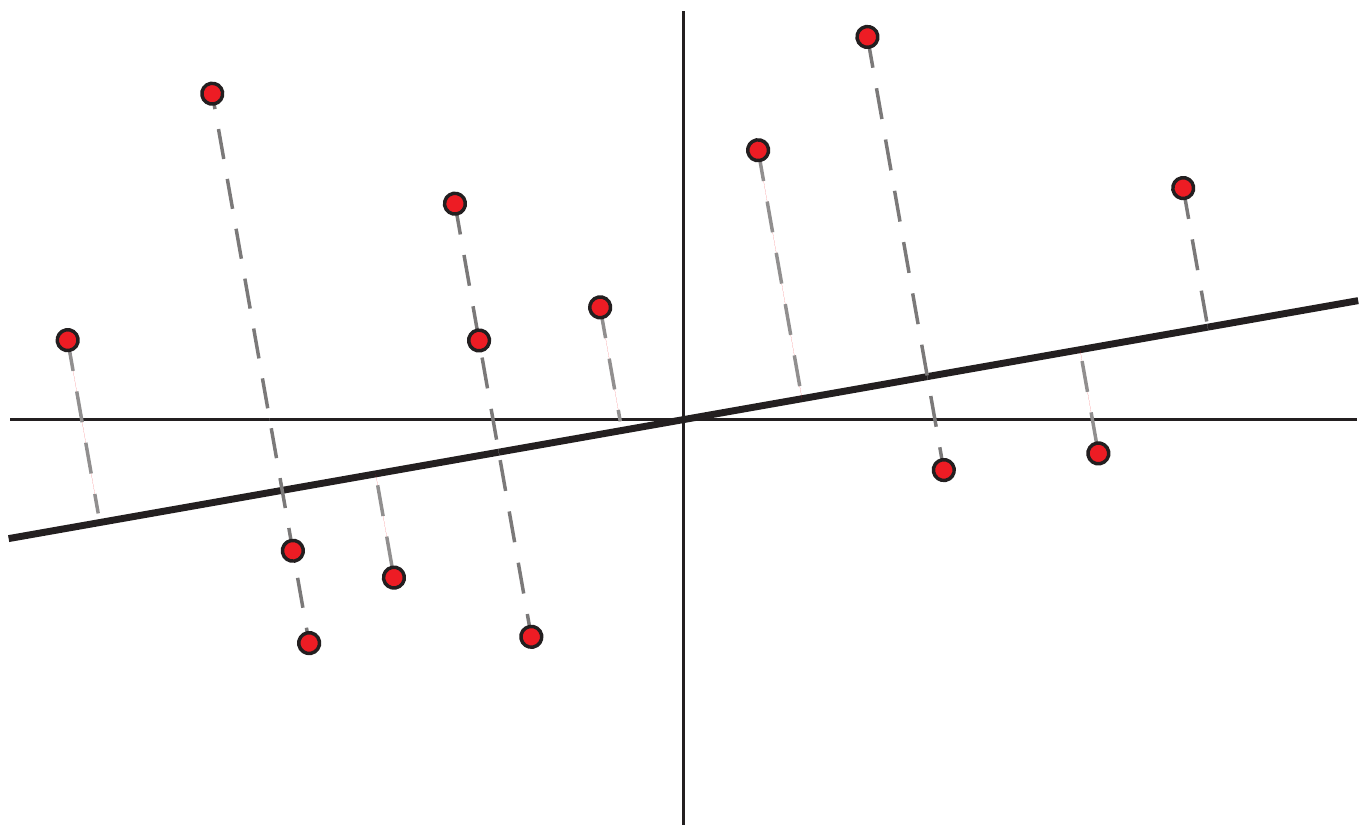}\caption{The projections of the eigenvalues on a critical ray. Here, $\ov{m}=(2,2,2,3,1,3,1)$.} \label{projection}\end{center}\end{figure}

\subsection{Statement of the theorem} \begin{theorem}\label{thm} We consider a system \eqref{equation} satisfying hypothesis \eqref{hypothesis}, and its Stokes matrices $C(j)$, $j=1, \dots, 2k,$ corresponding to the choice of $\mathbb{R}^+$ as starting ray. Let $\phi_1<\dots<\phi_{2ku}$ be the angles of the separating rays. Let $e^{i\theta}\R^+$, with $\theta\in (\phi_1,\phi_2) $, be a new starting ray such that the new sectors can be chosen as $\widetilde{S}_j=e^{i\theta} S_j$ (see Figure~\ref{turned_sectors}), and let $\widetilde{C}(j)$ be new Stokes matrices associated to the collection of sectors $\widetilde{S}_j$.
We suppose that the order of eigenvalues on the separating ray $R_1=e^{i\phi_1}\mathbb R^+$ is given by $\overline{m}= (s_1,r_2, \dots, s_{2m+1})$. Using a block notation $C(j)_{i,\ell}$ with $i,\ell\in \{1, \dots, 2m+1\}$, the size of the blocks corresponding to the partition of $n$ given by $\overline{m}$, then the $\widetilde{C}(j)$ can be chosen as
\begin{align} \begin{split} \widetilde{C}(j)&=P_{\overline{m}}\cdot \mathrm{diag}\left(I_{s_1}, C(j)_{2,2}^{-1}, I_{s_3}, \dots , C(j)_{2m,2m}^{-1}, I_{s_{2m+1}}\right)  \\
&\qquad\cdot C(j)\cdot\mathrm{diag}\left(I_{s_1}, C(j+1)_{2,2}, I_{s_3}, \dots , C(j+1)_{2m,2m}, I_{s_{2m+1}}\right)\cdot P_{\overline{m}}.\end{split} \end{align}
 \end{theorem}
 \begin{figure}\begin{center}
\includegraphics[width=6cm]{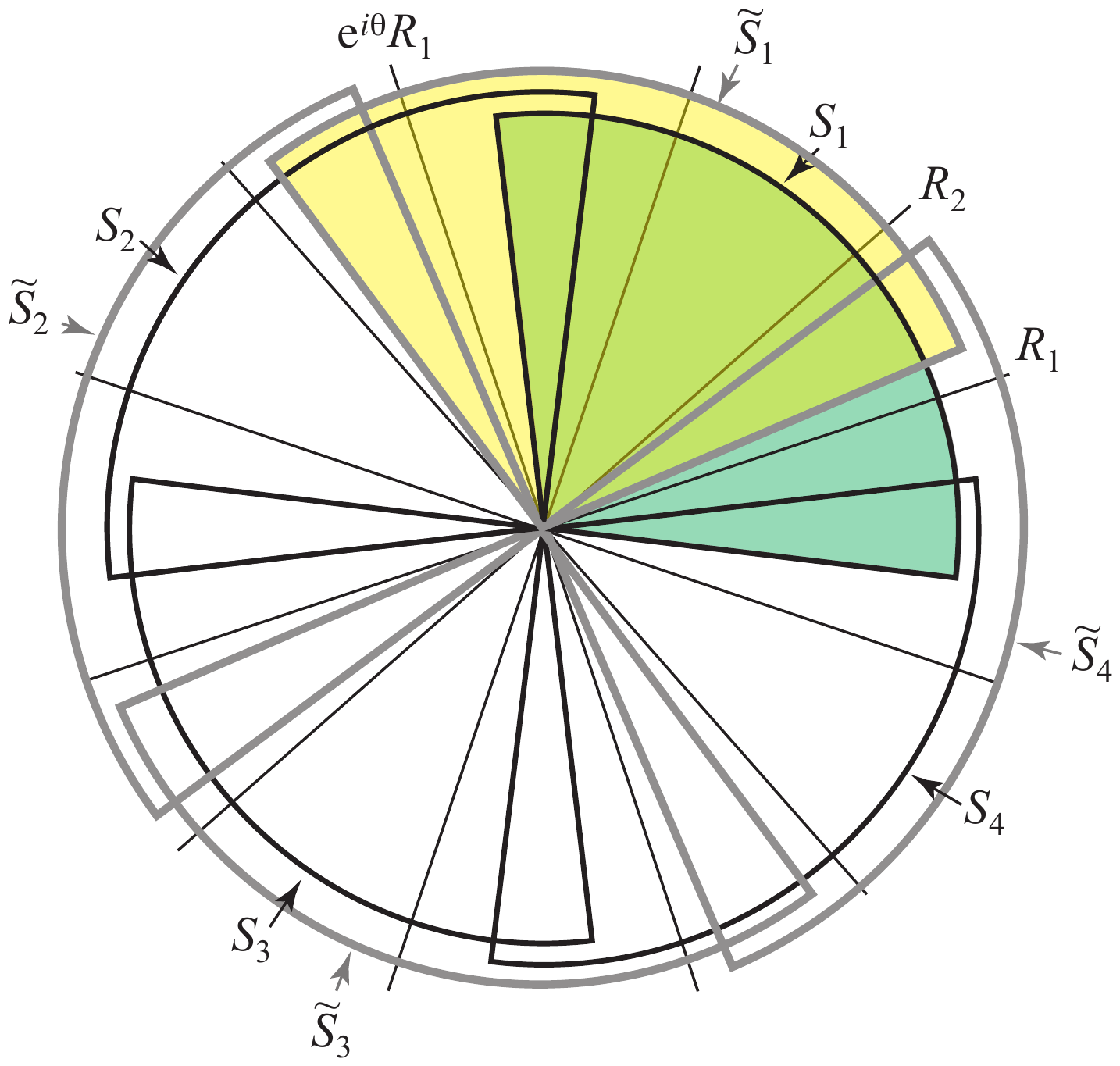}\caption{The sectors $S_i$ (with black boundary) and $\widetilde{S}_i$ (with grey boundary) when $k=2$. The other lines are the separating rays.} \label{turned_sectors}\end{center}\end{figure}

\begin{example}[Explicit computation in the case $k=1$ and $n=3$]

Consider the case where $R_1,R_2,R_3$ are distinct, with $R_j=e^{i\phi_j}\mathbb R^+$.  On $R_1$, let us suppose that the projections of $\lambda_2$ and $\lambda_3$ coincide. On $R_2$, the projections of $\lambda_1$ and $\lambda_3$ coincide, and on $R_3$ the projections of $\lambda_1$ and $\lambda_2$ coincide. Take starting rays $e^{i\theta_\ell}\mathbb R^+$, $\ell\in\{0,1,2,3\}$ such that
\begin{eqnarray}
0=\theta_0<\phi_1<\theta_1<\phi_2<\theta_2<\phi_3<\theta_3=\pi.
\end{eqnarray}
Let us write $C(1)=C^+= (c_{ij}^+)$ and $C(2)= C^-= (c_{ij}^-)$. Then choosing $R=e^{i\theta_\ell}\mathbb{R}^+$ as the starting ray, one gets a pair of Stokes matrices

{\small \begin{equation}
\begin{array}{r|cc}
R & C(1)=C^+ & C(2)=C^-\\
\hline
&&\\
\mathbb{R}^+ & \begin{pmatrix}
c_{11}^+& c_{12}^+ & c_{13}^+ \\
0 & c_{22}^+ & c_{23}^+ \\
0 & 0 & c_{33}^+
\end{pmatrix} & \begin{pmatrix}
c_{11}^- & 0 & 0 \\
c_{21}^- & c_{22}^- & 0 \\
c_{31}^- & c_{32}^- & c_{33}^- \\
\end{pmatrix} \\
&&\\
e^{i\theta_1}\mathbb{R}^+ & \begin{pmatrix}
c_{11}^+ & c_{13}^+c_{33}^- & c_{22}^-c_{12}^++c_{32}^-c_{13} ^+\\
0 & c_{33}^- & c_{32}^- \\
0 & 0 & c_{22}^-
\end{pmatrix} & \begin{pmatrix}
c_{11}^- & 0 & 0 \\
\frac{c_{22}^-c_{31}^--c_{32}^-c_{21}^-}{c_{22}^-c_{33}^-} & c_{33}^+ & 0 \\
\frac{c_{21}^-}{c_{22}^-} & c_{23}^+ & c_{22}^+
\end{pmatrix} \\
&&\\
e^{i\theta_2}\mathbb{R}^+ & \begin{pmatrix}
c_{33}^+ & \frac{c_{22}^-c_{31}^--c^-_{21}c^-_{32}}{c^-_{22}c^-_{33}} & \frac{c^-_{32}}{c^-_{33}} \\
0 & c_{11}^- & \frac{c^-_{22}c^+_{12}}{c^+_{11}} \\
0 & 0 & c_{22}^-
\end{pmatrix} & \begin{pmatrix}
c^-_{33} & 0 & 0 \\
c^-_{33}c^+_{13} & c^+_{11} & 0 \\
c^-_{33}\left(\frac{c^-_{21}c^+_{13}}{c^-_{22}}+c^+_{23}\right) & \frac{c^-_{21}c^+_{11}}{c^-_{22}} & c^+_{22}
\end{pmatrix} \\
&&\\
-\mathbb{R}^+ & \begin{pmatrix}
c_{33}^+ & \frac{c^-_{32}c^+_{22}}{c^-_{33}} & \frac{c^-_{31}c^+_{11}}{c^-_{33}} \\
0 & c_{22}^+ & \frac{c^-_{21}c^+_{11}}{c^-_{22}} \\
0 & 0 & c_{11}^+
\end{pmatrix} & \begin{pmatrix}
c^-_{33} & 0 & 0 \\
\frac{c^-_{33}c^+_{23}}{c^+_{22}} & c^-_{22} & 0 \\
\frac{c^-_{33}c^+_{13}}{c^+_{11}} & \frac{c^-_{22}c^+_{12}}{c^+_{11}} & c^-_{11}
\end{pmatrix}
\end{array}
\end{equation}}

Let us call the Stokes matrices $\overline{C}^\pm$ for $\theta_3=\pi$. One would expect that $\overline{C}^\mp$ would be equal to  $P_{0,3,0}C^\pm P_{0,3,0}$. This is not the case, but the difference comes from the fact that the matrices are only determined up to diagonal matrices. Indeed, 
% \begin{align}\begin{split}
% D\overline{C}^+D_*^{-1}&=\begin{pmatrix}
% c^-_{33} & 0 & 0 \\
% 0 & c^-_{22} & 0 \\
% 0 & 0 & c^-_{11}
% \end{pmatrix}\begin{pmatrix}
% c^+_{33} & 0 & 0 \\
% \frac{c^+_{33}c^-_{23}}{c^-_{22}} & c^+_{22} & 0 \\
% \frac{c^+_{33}c^-_{13}}{c^-_{11}} & \frac{c^+_{22}c^-_{12}}{c^-_{11}} & c^+_{11}
% \end{pmatrix}\begin{pmatrix}
% c^+_{33} & 0 & 0 \\
% 0 & c^+_{22} & 0 \\
% 0 & 0 & c^+_{11}
% \end{pmatrix}^{-1}\\
% &=\begin{pmatrix}
% c^-_{33} & 0 & 0 \\
% c^-_{23} & c^-_{22} & 0 \\
% c^-_{13} & c^-_{12} & c^-_{11}
% \end{pmatrix}=P_{0,3,0}\,C^-\,P_{0,3,0}.
% \end{split} \end{align}
% Similarly we can show that  $D_*\overline{C}^-D^{-1}= P_{0,3,0}
% \,C^+\,P_{0,3,0}$.

\begin{align}\begin{split}
D\overline{C}^+D_*^{-1}&=\begin{pmatrix}
c^-_{33} & 0 & 0 \\
0 & c^-_{22} & 0 \\
0 & 0 & c^-_{11}
\end{pmatrix} \begin{pmatrix}
c_{33}^+ & \frac{c^-_{32}c^+_{22}}{c^-_{33}} & \frac{c^-_{31}c^+_{11}}{c^-_{33}} \\
0 & c_{22}^+ & \frac{c^-_{21}c^+_{11}}{c^-_{22}} \\
0 & 0 & c_{11}^+
\end{pmatrix} \begin{pmatrix}
c^+_{33} & 0 & 0 \\
0 & c^+_{22} & 0 \\
0 & 0 & c^+_{11}
\end{pmatrix}^{-1}\\
&=\begin{pmatrix}
c^-_{33} & c^-_{32}  & c^-_{31} \\
0& c^-_{22} & c^-_{21}\\
0 &0  & c^-_{11}
\end{pmatrix}=P_{0,3,0}\,C^-\,P_{0,3,0}.
\end{split} \end{align}
Similarly we can show that  $D_*\overline{C}^-D^{-1}= P_{0,3,0}
\,C^+\,P_{0,3,0}$.
\end{example}

\subsection{Proof of the theorem} 

The first step is the reduction to the case $m=1$ (see \eqref{def:m} for the definition of $m$). This comes from the fact that the phenomena at each block of eigenvalues having equal projections on the critical ray $\ov{R}_1= e^{ik\phi_1}\R^+$ are independent. Suppose for instance that the theorem is proved when $m=1$, and consider the case  $m=2$. 
Consider a perturbation of the system in which we multiply by $e^{i\eps}$, 
for some small $\eps$, the eigenvalues of the second block of eigenvalues which have equal projection on the critical ray $\ov{R}_1$. Then, when $\eps$ is real, small  and nonzero, we have two critical rays $\ov{R}_1$ et $\ov{R}_1'= e^{i\eps}R_1$,  and two separating rays $R_1$ and $R_1'= e^{i\frac{\eps}k}R_1$. 
For nonzero $\eps$, the passage from the starting ray $\mathbb R^+$  to the starting ray $e^{i\theta} \mathbb R^+$ is obtained by applying Theorem~\ref{thm} twice: when $\eps>0$, we first consider the change in the Stokes matrices when passing $R_1$ using Theorem~\ref{thm};  then we change $x\mapsto xe^{-i(\phi_1+\frac{\eps}2)}$ and pass $R_1'$ using Theorem~\ref{thm} a second time. When $\eps<0$, the passages are in the reverse order. The two passages commute and the final result is independent of the sign of $\eps$. Moreover, the construction of the Stokes matrices shows that they depend analytically on the eigenvalues. Then  the limit passage when $\eps=0$ is the composition of the passages for each block of eigenvalues. The same reasoning can be done for any $m\geq2$. 

Hence, from now on, we treat the case $m=1$, i.e. $n= s_1+r_2+s_3$ and the eigenvalues $\lambda_j$  with $j\in[s_1+1, s_1+r_2]$ have equal projections on the critical ray $\ov{R}_1=e^{ik\phi_1}\mathbb R^+$.

Let $$W(x)= \text{diag} (\omega_1(x), \dots , \omega_n(x))$$  be the diagonal fundamental matrix solution of the normal form of \eqref{equation} at 0. Hypothesis \eqref{hypothesis} implies that 
\begin{eqnarray}
\omega_1\prec \omega_2\prec ...\prec \omega_n
\end{eqnarray}
on $\mathbb{R}^+$ and everywhere on $S_1\cap S_{2k}$. As a matter of fact, this ordering is precisely why we take \eqref{hypothesis} as a hypothesis. This order is respected on every thin intersection $S_{2j}\cap S_{2j+1}$ and is completely reversed on $S_{2j+1}\cap S_{2j+2}$. A direct consequence is that the Stokes matrices are alternatively upper and lower triangular.\\
\\
This whole construction depends on the choice of $\mathbb{R}^+$ as our starting ray, but this choice is not canonical.

\subsubsection{New order on the eigenvalues}

We start by describing the changes induced by choosing $\widetilde{R}=e^{i\theta}\mathbb{R}^+$ instead of $\mathbb{R}^+$ as starting ray commanding the  order of the eigenvalues. 

\begin{lemma} Let $R_1=e^{i\phi_1}\mathbb R^+$ be the first separating ray by order of increasing argument. Let us suppose that on the associated critical ray $\ov{R}_1=e^{ik\phi_1}\mathbb R^+$  precisely the following signed projections are equal \begin{equation}\mathfrak{R}(e^{-ik\phi_1} \lambda_{j_1})= \dots =  \mathfrak{R}(e^{-ik\phi_1} \lambda_{j_{r_2}}),\label{equal_signed_proj}\end{equation} and the others are distinct. If $j_1=s_1+1$, then $j_2=s_1+2, \dots,j_{r_2}=s_1+r_2$. Moreover, for $\phi_1<\theta<\phi_2$, the new order of the eigenvalues on $e^{i\theta}\R^+$ is obtained by completely reversing the order of the eigenvalues at positions $s_1+1$ to $s_1+r_2$ and leaving the others as they were. \end{lemma}
\begin{proof} 
The signed projections on $e^{ik\theta}\mathbb{R}^+$ depend continuously on $\theta$, implying that if $j_1=s_1+1$, then $j_2=s_1+2, \dots,j_{r_2}=s_1+r_2$. Let us call $f(\theta,\lambda_j) = \mathfrak{R}(e^{-ik\theta}\lambda_j)$. Then $\frac{\partial}{\partial \theta} f(\theta,\lambda_j)= k\mathfrak{I} (e^{-ik\theta}\lambda_j).$ Since the $\lambda_\ell$ are distinct, and because \eqref{equal_signed_proj} is satisfied, it follows that the $\mathfrak{I} (e^{-ik\phi_1}\lambda_{j_m})$ are distinct. Since $\mathfrak{R}(e^{-ik\theta} \lambda_{j_1})<\dots < \mathfrak{R}(e^{-ik\theta} \lambda_{j_s})$ for $0\leq \theta<\phi_1$, then $\mathfrak{I} (e^{-ik\phi_1}\lambda_{j_1})> \dots >\mathfrak{I}(e^{-ik\phi_1}\lambda_{j_s})$, from which the conclusion follows. \end{proof} 

\begin{corollary} Let us suppose that on $R_1$, $m$ subsets of consecutive eigenvalues have equal order. Then for $
\phi_1<\theta<\phi_2$, the new order of the eigenvalues on $\widetilde{R}$ is obtained from the order in \eqref{hypothesis} by completely reversing the order in each group. \end{corollary}

\subsubsection{New sectors $\widetilde{S}_j$}

Under  \eqref{hypothesis}, let $0\leq \phi_1< \dots <\phi_N<2\pi $ be the separating rays, and let $$\delta= \min\{\phi_2-\phi_1, \dots, \phi_N-\phi_{N-1},\phi_1,2\pi -\phi_N\}>0.$$ The sectors  $S_j$ can be chosen so that 
$$S_j=\left\{x\: ;\: |x|<r, \arg(x) \in\left( \frac{(j-1)\pi}{k}-\frac{\delta}4,\frac{j\pi}{k}+\frac{\delta}4\right)\right\}.$$ This definition allows simply defining the new sectors as \begin{equation}
\widetilde{S}_j=e^{i\left(\phi_1+\frac{\delta}{2}\right)}S_j.
\end{equation}

\subsubsection{New Stokes matrices}

Let us now describe the Stokes matrices of the system using the starting ray $\widetilde{R}= e^{i\theta}\R^+$, i.e. the sectors $\widetilde{S}_j$, which we will denote by $\widetilde{C}(j)$. For that purpose we need to find, for each $j$, a new fundamental matrix solution $\widetilde{W}_j$ on $\widetilde{S}_j$, which exhibits the correct order of flatness on $\widetilde{S}_{j-1}\cap \widetilde{S}_{j}$ and $\widetilde{S}_{j}\cap \widetilde{S}_{j+1}$. Then we will have
\begin{equation}\widetilde{C}_j= \widetilde{W}_j^{-1}\widetilde{W}_{j+1}.\label{obtaining_Stokes}\end{equation}

We claim that such a new fundamental matrix solution can be taken as \begin{equation}
\widetilde{W}_j=W_j\begin{pmatrix} 
I_{s_1}& 0 & 0 \\
0 & C(j)_{2,2} & 0 \\
0 & 0 & I_{s_3}
\end{pmatrix}P_{s_1,r_2,s_3}.\label{new_Stokes} 
\end{equation}
Without loss of generality we can suppose that $j=1$, the other cases being similar. In this case, we need only prove that
$\widetilde{W}_1$
is a fundamental matrix solution, the columns of which satisfy 
\begin{equation}\begin{cases}
\widetilde{w}_1\prec \dots \prec\widetilde{w}_n,
&\text{on}\:\:\widetilde{S}_{2k}\cap\widetilde{S}_{1},\\
\widetilde{w}_1\succ \dots\succ\widetilde{w}_n,
&\text{on}\:\:\widetilde{S}_{1}\cap\widetilde{S}_{2}.\end{cases} \label{right_order} \end{equation}
The proof uses the following facts:
\begin{itemize} 
\item We know that such a fundamental matrix solution exists. This comes from the sectorial normalization theorem (\cite{S}) for the system \eqref{equation} after a change $x\mapsto e^{-i\theta} x$.  
\item Moreover, we know that a matrix $\widetilde{W}_j$, the columns of which satisfy \eqref{right_order}, is unique up to right multiplication by a diagonal matrix.
\end{itemize} 
Hence, as soon as we show that the choice \eqref{new_Stokes} is the only possible choice (up to right multiplication by a diagonal matrix) that could meet the constraints \eqref{right_order}, then we are sure  that it indeed does satisfy the constraints. 

We discuss what occurs when we cross a separating ray. We say that we are \emph{before} (resp. \emph{after}) the separating ray $e^{i\phi}\mathbb R^+$ if we are in a region $\arg(x)<\phi$ (resp. $\arg(x)>\phi$). Also, note that each sector $S_j$ or $\widetilde{S}_j$ contains exactly one separating ray for each pair of eigenvalues. For instance, since $R_1$ is the separating ray inside $S_1$ and $\widetilde{S}_{2k}$ for any pair of eigenvalues within $\{\lambda_{s_1+1}, \dots, \lambda_{s_1+r_2}\}$, then  $e^{i\left(\phi_1+ \frac{(j-1)\pi}{k}\right)}\mathbb R^+=e^{\frac{i(j-1)\pi}{k}}R_1$ is a separating ray for the same pair of eigenvalues  inside $S_j$, and also inside the new sector $\widetilde{S}_{j-1}$. (This is a particular case of the general fact that if $R_p$ is some separating ray for some subset of eigenvalues inside a sector $S_\ell$, then $e^{\frac{is\pi}{k}}R_p$ is a separating ray for the same subset of eigenvalues inside $S_{\ell+s}$.)

It is straightforward that the solutions $\widetilde{w}_\ell=w_\ell$ for $\ell=1,..,s_1$, are adequate because $R_1$ is a separating ray only for pairs of eigenvalues among $\{\lambda_{s_1+1},\dots,\lambda_{s_1+r_2}\}$. Hence, $\widetilde{w}_1\prec \dots \prec\widetilde{w}_{s_1}$ on $\widetilde{S}_{2k}\cap\widetilde{S}_{1}$ since it is the case on $S_{2k}\cap S_1$. Also, in $\widetilde{S_1} \cap \widetilde{S_2}$, we have $\widetilde{w}_1\succ \dots\succ\widetilde{w}_{s_1}$ since we passed one separating ray for each pair of  eigenvalues among $\lambda_1, \dots, \lambda_{s_1}$.

Similarly, it is straightforward that the solutions $\widetilde{w}_\ell=w_\ell$ for indices $\ell=s_1+r_2+1,\dots,n$ are adequate.

Moreover, from \eqref{asympt_expansion} it is clear that on  $\widetilde{S}_{2k}\cap\widetilde{S}_{1}$, for $ s_1+1\leq j\leq s_1+r_2$,
\begin{equation}\widetilde{w}_1\prec\widetilde{w}_2 \prec \dots\prec\widetilde{w}_{s_1} \prec \widetilde{w}_{j}\prec \widetilde{w}_{s_1+r_2+1} \prec\dots\prec \widetilde{w}_n.\label{partial_asympt}\end{equation}
This comes from the fact that $\widetilde{S}_{2k}\cap\widetilde{S}_{1}\subset S_1$, from \eqref{asympt_expansion}, and from the fact that we have only crossed the separating ray $R_1$.  We have the asymptotic order reverse to \eqref{partial_asympt}  on $\widetilde{S_1} \cap \widetilde{S_2}$. Indeed,  $\widetilde{S_1} \cap \widetilde{S_2}\subset S_2$ is located after the separating ray $e^{i\frac{\pi}{k}}R_1$ (the second separating ray for the pairs of eigenvalues in the block), and before the second separating rays for the other pairs of eigenvalues.

It remains to compare the solutions $\widetilde{w}_\ell$ for $\ell\in\{s_1+1,\dots,s_1+r_2\}$.
If $C(1)=\left( c_{l,m}\right)_{l,m=1}^n$, then the linear combination
\begin{equation}
\overline{w}_\ell=w_{s_1+1}c_{s_1+1,\ell}+\dots+w_{s_1+r_2}c_{s_1+r_2,\ell}\label{def_w_ell}
\end{equation}
provides exactly vectors that have the same order of flatness with respect to $w_j$ for $j\notin\{s_1+1, \dots s_1+r_2\}$. We claim that they are ordered as: 
\begin{equation}\begin{cases}
\overline{w}_{s_1+1}\succ \dots \succ\overline{w}_{s_1+r_2},
&\text{on}\:\;\widetilde{S}_{2k}\cap\widetilde{S}_{1},\\
\overline{w}_{s_1+1}\prec \dots\prec\overline{w}_{s_1+r_2},
&\text{on}\:\;\widetilde{S}_{1}\cap\widetilde{S}_{2}.\end{cases} \label{claim}\end{equation}
Hence, reordering the vectors by letting $$
\widetilde{w}_{s_1+\ell} =\overline{w}_{s_1+r_2-\ell}$$ for $\ell\in\{1, \dots r_2\}$, which corresponds to applying the permutation matrix $P_{0,r_2,0}$ to the part of the fundamental matrix solution corresponding to $w_{s_1+1}, \dots, w_{s_1+r_2}$ (i.e. $P_{(s_1,r_2,s_3)}$ to the full $n\times n$ fundamental matrix solution), yields the theorem. 

Let us now prove the claim \eqref{claim}. The first part on $\widetilde{S}_{2k}\cap\widetilde{S}_{1}$ follows from  \eqref{asympt_expansion} and the fact that we are after $R_1$. To derive the second conclusion, let $\{\widehat{w}_1, \dots ,\widehat{w}_n\}$ be the basis given by the fundamental matrix solution $W_2$ on $S_2$. Then the order of flatness of $\overline{w}_{s_1+1}, \dots,\overline{w}_{s_1+r_2}$ on $\widetilde{S}_{1}\cap\widetilde{S}_{2}$ is the same as that of $\widehat{w}_{s_1+1}, \dots ,\widehat{w}_{s_1+r_2} $ on  $S_2\cap S_3$ because we passed $e^{\frac{\pi i}{k} }R_1$.
Indeed, the difference 
$\widehat{w}_\ell-\overline{w}_\ell$ is a linear combination of the $w_i$ for $i>s_1+r_2$: 
$$\widehat{w}_\ell-\overline{w}_\ell = \sum_{i>s_1+r_2}b_{i\ell} w_i,$$ since $C(1)$ is lower triangular.  From \eqref{asympt_expansion}, all these $w_i$ are flatter than $\overline{w}_\ell$ defined in \eqref{def_w_ell} after we have crossed a  separating ray associated to them, which is the case on $\widetilde{S}_{1}\cap\widetilde{S}_{2}$. Hence $w_i\prec\overline{w}_\ell$ on $\widetilde{S}_1\cap\widetilde{S}_2$. This yields $\widehat{w}_{s_1+1}\prec \dots\prec\widehat{w}_{s_1+r_2}$ on $\widetilde{S}_{1}\cap\widetilde{S}_{2}\subset S_2$ since we are in $S_2$ and after $e^{\frac{\pi i}{k} }R_1$. Hence $\overline{w}_{s_1+1}\prec \dots\prec\overline{w}_{s_1+r_2}$ on $\widetilde{S}_{1}\cap\widetilde{S}_{2}$.
 \hfill $\Box$

\end{document}